\documentclass[10pt]{amsart}
\usepackage[margin=1.1in]{geometry}

\input{arxiv2.sty}
\newcommand{\St}{\mathrm{St}}
\newcommand{\ST}[2]{\St_{#1}({#2})}
\newcommand{\Gr}[2]{\mathrm{Gr}_{#1}({#2})}
\newcommand{\CH}{\mathrm{CH}}
\newcommand{\CHW}{\widetilde{\CH}{}}
\newcommand{\spec}{\mathrm{Spec}\,}

\newcommand{\BSp}{\operatorname{B}\!\Sp}
\newcommand{\STsp}[2]{\St^{\Sp}_{#1}({#2})}
\newcommand{\Grsp}[2]{\mathrm{Gr}^{\Sp}_{#1}({#2})}

\theoremstyle{plain}
\newtheorem{problem}{Problem}

\title[Efficient generation of projective modules: a motivic view]{Efficient generation of projective modules: \\a motivic view}
\author{Aravind Asok, Morgan Opie, Brian Shin, and Tariq Syed}
\date{\today}

\begin{document}

\begin{abstract}
	Assume $k$ is a field and $R$ is a smooth $k$-algebra of dimension $d$.  If $P$ is a projective module of rank $r$, then it is well-known that $P$ can be generated by $r+d$-elements (Forster--Swan).  Under suitable assumptions on $r$ and $d$, we investigate obstructions to generation of $P$ by fewer than $r+d$ elements using motivic homotopy theory.  For example, we observe that a quadratic enhancement of the classical Segre class obstructs generation by $r+d-1$ elements, whether or not $k$ is algebraically closed, generalizing old results of M.P. Murthy.  Along the way, we also establish efficient generation results for symplectic modules.  
\end{abstract}

\maketitle
\section{Introduction}
Assume $R$ is a commutative ring, and $M$ is a finitely generated $R$-module.  Write $\nu(M)$ for the minimal number of generators of $M$.  In 1961, R. Swan posed the following problem \cite[p. 272 Problem]{SwanProjModVB}, which he attributed to J.-P. Serre.  

\begin{problem}[Serre--Swan]\label{Serre-swan}
	If $R$ is a commutative Noetherian ring whose maximal ideal spectrum is a Noetherian topological space of dimension $d$, and $M$ is a rank $r$ projective module, then is $\nu(M) \leq r + d$?
\end{problem}

There are many variants of this kind of problem: one could drop the hypothesis that $M$ is projective, retaining only the condition that it is a finitely generated module; one could drop the hypothesis that $R$ is Noetherian, retaining only enough structure to define dimension, etc. When $R$ is Noetherian ring of Krull dimension $d$, if the localizations $M_{{\mathfrak m}}$ can be generated by $r$ elements for all maximal ideals ${\mathfrak m}$ of $R$, then O. Forster showed that $\nu(M) \leq r + d$ \cite[Satz 1]{Fo}.  Shortly thereafter, R. Swan established the same result eliminating the Noetherian hypotheses on $R$: if the maximal ideal spectrum of $R$ is a Noetherian topological space of dimension $d$, and if for each maximal ideal ${\mathfrak m}$ of $R$ the module $M_{{\mathfrak m}}$ is generated by $r$ elements, then $\nu(M) \leq r+d$ \cite[Theorem 1]{Sw}.  

The original Serre--Swan problem was explicitly based on an analogy with a corresponding topological situation, so let us rephrase the algebraic problem in geometric terms.  Indeed, the Grassmannian $\Gr{r}{n}$ represents the functor on commutative rings that, to a commutative ring $R$, assigns the set of $n$-generated projective $R$-modules of rank $r$.  If $M$ is an $n$-generated projective $R$-module, then asking whether $M$ can be generated by $n' < n$ elements amounts to asking whether the classifying map $\spec R \to \Gr{r}{n}$ associated with $M$ can be lifted along a standard map $\Gr{r}{n'} \to \Gr{r}{n}$.  If $n'$ can be taken to be smaller than $r+d$, we will say that $M$ is {\em efficiently generated}. 

One celebrated efficient generation result was established by M.P. Murthy, who observed that if $R$ is furthermore a regular $k$-algebra with $k$ an algebraically closed field, then there are geometric obstructions to efficient generation of $M$.  In the context we consider, Murthy's results imply: if $M$ is a projective $R$-module of rank $r$, then there is an associated {\em Segre class} $s_0(M) \in \CH_0(\Spec R)$ whose vanishing is necessary and sufficient for $M$ to be efficiently generated \cite[Theorem 5.2 and Corollary 5.3]{Mu}.

Our aim here is to further analyze such efficient generation results in the case of smooth affine $k$-algebras $R$ over a field $k$ (not necessarily algebraically closed).  In this context, obstruction-theoretic techniques \cite{Morel} in the Morel--Voevodsky motivic homotopy theory \cite{MV99} can be brought to bear on the problem.  Write $\mathcal H(k)$, the homotopy category of motivic spaces over $k$.

Affine representability results \cite{AHW1,AHW2} imply that: 
\begin{itemize}
	\item finitely generated projective $R$-modules of rank $r$ are classified up to isomorphism by $\mathbb A^1$-homotopy classes of maps $\spec R \to \BGL_r = */\GL_r$ 
	\item $n$-generated projective $R$-modules of rank $r$ are classified up to isomorphism by $\mathbb{A}^1$-homotopy classes of maps $\spec R \to \Gr{r}{n}$.
\end{itemize}
Moreover, there is a canonical map $\Gr{r}{n} \to \BGL_r$ classifying the universal rank $r$ vector bundle on $\Gr{r}{n}$.  

In this context, efficient generation questions amount to analyzing the following lifting problem:
\[
\xymatrix{
	& \Gr{r}{n} \ar[d] \\
\spec R \ar[r]\ar@{.>}[ur]^{\exists ?} & \BGL_r.
}
\]
Obstruction theory for the Moore--Postnikov factorization of the vertical map then yields an inductively defined sequence of obstructions whose vanishing is necessary and sufficient for the existence of the dotted arrow in the diagram.  The relevant obstructions are controlled by the homotopy fiber of the map $\Gr{r}{n} \to \BGL_r$, which is identified as the Stiefel variety $\ST{r}{n} =\GL_n / \GL_{n-r}$ (cf. \Cref{prop:fib_stiefel_pf}).  

With these preliminaries at hand, we can now state our main results.  As a warm-up, one observes a mild improvement of the Forster--Swan bound follows immediately from connectivity estimates for $\ST{r}{n}$ (the ${\mathbb A}^1$-cohomological dimension appearing in the statement is bounded above by the Krull dimension, but can be strictly smaller).

\begin{thmx}[{cf. \Cref{htpy:Fo-Sw} }]\label{thm:Z}
Let $X = \spec R$ be a smooth affine variety of ${\mathbb A}^1$-cohomological dimension at most $d \geq 2$ over a perfect field $k$. Any finitely generated projective $R$-module of rank $r$ can be generated by $r+d$ elements as an $R$-module. 
\end{thmx}

The precise obstructions for a finitely generated projective module of rank $r$ over a smooth affine algebra of $\A^1$-cohomological dimension at most $d$ to be generated by $r+d-1$ or even by $r+d-2$ elements can also be identified, at least under certain hypotheses. We are able to generalize Murthy's result:

\begin{thmx}[{cf. \Cref{thm:different-murthy}}]\label{thm:A}Let $X=\spec R$ be a smooth affine variety of $\A^1$-cohomological dimension at most $d\geq 3$ over a perfect field $k$.  Let $M$ be a rank $r$ projective module over $R$, where $r\geq 2$. 
\begin{itemize}
\item If $d$ is even, $M$ can be generated by $d+r-1$ generators if and only if a {\em Segre class} of $M$ vanishes in the $d$-th Chow group of $X$.
\item If $d$ is odd and $M \oplus Q \simeq R^{d+r}$ for $Q$ a projective module of rank $d$, then $M$ can be generated by $d+r-1$ generators if and only if an {\em Euler class} of $Q$ is zero in the $d$-th twisted Chow--Witt group of $X$. Such a $Q$ always exists, and this condition is independent of its choice.
\end{itemize} 
\end{thmx}
Under additional hypotheses on the field, we prove:
\begin{thmx}[{cf. \Cref{thm:line-bundle-generation,cor:murthy-segre-suff-quad-closed}}]\label{thm:B}Let $X=\spec R$ be a smooth affine variety of dimension at most $d\geq 3$ over a perfect field $k$. Suppose also that $k$ has $2$-cohomological dimension at most $1$. A rank $r$ projective module $P$ can be generated by $r+d-1$ elements as an $R$-module if and only if the Segre class $s_d(P)$ vanishes.
\end{thmx}
The key step to prove \Cref{thm:A} is our identification of the first possibly nontrivial obstruction to $n$-generation of a rank $r$ module $M$ over $X=\spec R$ a smooth affine variety, without any reference to dimension. When the ${\mathbb A}^1$-cohomological dimension of $X$ is greater than $n-r+1$, there is a secondary obstruction to consider. We can identify this obstruction under suitable hypotheses.

\begin{thmx}[{cf. \Cref{char-zero-segre-vanish-rd2,thm:rd2-alg-closed-char-neq-2}}]\label{thm:C}  Let $X = \spec R$ be a smooth affine variety of dimension at most $d\geq 4$ over an algebraically closed field $k$. Let $M$ be a projective $R$-module of rank $r\geq 2$. 
\begin{itemize}
\item If $k$ is characteristic zero, then $M$ can be generated by $r+d-2$ elements if and only if the $d$-th and $(d-1)$-st Segre classes of $M$ vanish. 
\item If $k$ has characteristic not equal to zero but $H^d(X,\pia_{d-1}(\A^{d-1}\setminus 0))=0$, then the same conclusion holds.
\end{itemize}
\end{thmx}

\begin{remark} \Cref{thm:Z} is also true for $d=1$; \Cref{thm:A} and \Cref{thm:B} hold in the case $r\geq 2$ and $d=2$; and \Cref{thm:C} remains true for $r \geq 2$ and $d=3$. However, the proofs require some modified technical arguments, which we omit from this article for brevity. See \Cref{rmk:rr1-discussion} for discussion.
\end{remark} 
The methods used to prove \Cref{thm:A,thm:B,thm:C} can be adapted to prove results for symplectic bundles. As such, we prove a symplectic Forster--Swan theorem.

\begin{thmx}[{cf. \Cref{thm:symplectic-F-S}}]\label{thm:thm:D}
Let $X = \spec R$ be a smooth affine variety of $\A^1$-cohomological dimension at most $d \geq 2$ over a perfect field $k$. Let $M$ be a symplectic module of rank $2r$. Then $M$ is a direct summand of a hyperbolic symplectic $R$-module of rank $2r+2\lfloor\frac{d}{2}\rfloor$.
\end{thmx} 

We also identify some obstructions to expressing a symplectic module $M$ as a summand of a lower rank hyperbolic $R$-module in \Cref{thm:symplectic-odd-further-reduction,cor:symplectic-even-further-reduction}. These results depend on the parity of the (${\mathbb A}^1$-cohomological) dimension of $X$.

\subsection{Related literature}\label{subsec:lit} As already indicated, there are many variants of \Cref{Serre-swan} above. While this paper presents a motivic approach to efficient generation of projective modules and symplectic modules, work by many authors studies efficient generation in other cases. For example, Eisenbud and Evans proved a strengthening of the Forster--Swan theorem that considers how different generating sets are related \cite[Theorem B]{eisenbudevans73}. %while efficient generation problems for ideals in polynomial rings or related rings have also been extensively studied.
Extensive results have been obtained for ideals in polynomial rings in \cite{add1,add2} and \cite{add3}, and for ideals in overrings of polynomial rings in \cite{add4}. More recently, efficient generation results were proven for ideals in geometric subrings of polynomial rings in \cite{add5} and for ideals in discrete Hodge algebras in \cite{add6} and \cite{add7}.

\subsection{Paper outline}
\Cref{sec:homotopy_sheaves} is devoted to studying aspects of the Grassmannians, Stiefel manifolds and related spaces, both in the classical and symplectic settings, from the standpoint of motivic homotopy theory.  In particular, \Cref{sec:fiber} describes fiber sequences relating these spaces, \Cref{subsec:stablerange} uses these fiber sequences in conjunction with connectivity/weak-cellular estimates to establish stable ranges for homotopy sheaves, while \Cref{subsec:actual_computations} and \Cref{sebsec:seq-add} combine these results to describe some non-stable $\mathbb A^1$-homotopy sheaves of (symplectic) Stiefel varieties.  

In \Cref{total-obstructions}, we leverage the computations just described to establish our main results on efficient generation of projective modules. \Cref{subsec:absolute} recalls basic ingredients of obstruction theory involving Moore--Postnikov towers, \Cref{subsec:classical-invariants} briefly recalls some aspects of the theory of Euler classes, \Cref{subsec:generation-line-bundles} presents a slight detour to describe efficient generation of invertible modules, \Cref{subsec:first-obstr-classical} identifies the primary obstruction to efficient generation in terms of a quadratically enhanced Segre class, while \Cref{further-reduction} analyzes further reduction of the number of generators in terms of vanishing of secondary obstructions.  Finally, in \Cref{sec:sym-forster-swan}, we adapt the methods from the previous section to prove results for efficient generation of symplectic modules.

\subsection{Preliminaries and conventions}\label{subsec:prelims}
Fix a field $k$. We write $\operatorname{Sm}_k$ for the category of separated, smooth $k$-schemes. Write $\operatorname{Sh}_{\operatorname{Nis}}(\operatorname{Sm}_k)$ for the infinity category of Nisnevich sheaves of spaces on $\operatorname{Sm}_k$. The infinity category of motivic spaces $\operatorname{Spc}(k)$ is the full subcategory of $\operatorname{Sh}_{\operatorname{Nis}}(\operatorname{Sm}_k)$ spanned by $\A^1$-local sheaves, i.e., those Nisnevich sheaves of spaces  $\mathcal X$ such that for all $U \in \operatorname{Sm}_k$, the map $\mathcal X(U) \to \mathcal X(U \times \A^1)$ induced by projection $U \times \A^1 \to U$ is an equivalence. Throughout this paper, we work in the homotopy category $\mathcal H(k)$ of motivic spaces over a perfect field $k$, or the pointed version.  We refer the reader to \cite{MV99}.

We additionally make note of the following conventions and definitions:
\begin{enumerate}
 \item Given a pointed motivic space $(\mathcal X, x_0)$, we write $\pia_i(\mathcal X,x_0)$ for the Nisnevich sheafification of the presheaf of groups (or sets if $i=0$) on $\operatorname{Sm}_k$ given by $U \mapsto \pi_i\mathcal (X(U),x_0)$. If $\mathcal X\in \operatorname{Spc}_k$ is $\mathbb A^1$-simply connected, meaning $\pia_0(\mathcal X)=*$ and $\pia_1(\mathcal X)=*$, we write $\pia_i(\mathcal X)$ for its homotopy sheaf with respect to any basepoint (cf. \cite[\S3.2]{MV99} for discussion of homotopy sheaves). 
\item We use the word {\em fiber} exclusively to refer to the $\A^1$-homotopy fiber of a map of (pointed) motivic spaces.
\item We will frequently be concerned with {\em strictly $\A^1$-invariant sheaves} in the sense of \cite[\S6]{Morel} or \cite{stronglystrictly}. Such sheaves include the homotopy sheaves $\pia_i(\mathcal X)$ for $\mathcal X$ a motivic space and $i\geq 2$. 
\item Other examples of strictly $\A^1$-invariant sheaves include the $i$-th Milnor and Milnor--Witt $K$-theory sheaves $\KM_i$ and $\KMW_i$, which will feature prominently.
\item Given a strictly $\A^1$-invariant sheaf $\mathbf P$ and a smooth affine scheme $X$, we will frequently work with the $i$-th Nisnevich cohomology of $X$ with coefficients in $\mathbf P$, denoted $H^i(X,\mathbf P)$.  If $\mathbf P$ is moreover equipped with a $\mathbb G_m$-action and $\mathcal L$ is a line bundle over $X$, we may also consider $H^i(X,\mathbf P(\mathcal L))$, the $i$-th Nisnevich cohomology of $X$ with coefficients in $\mathbf P$ twisted by $\mathcal L$.
\item Let $X$ be a smooth $k$-scheme. We define the {\em $\A^1$-cohomological dimension of $X$} to be the smallest intger $d$ such that, for all $n > d$ and all strictly $\A^1$-invariant Nisnevich sheaves of abelian groups on $\operatorname{Sm}_k$, the (Nisnevich) cohomology group $H^{n}(X,\mathcal A)$ is zero. 
\item For $X$ a smooth $k$-scheme, we write $\CH^i(X)$ for the $i$-th Chow group group of codimension $i$ cycles on $X$. We write $\CHW^i(X)$ for the $i$-th Chow--Witt group of $X$.
\item Given motivic spaces $\mathcal X$ and $\mathcal Y$, we write $[\mathcal X,\mathcal Y]$ for $\A^1$-homotopy classes of maps from $\mathcal X$ to $\mathcal Y$, i.e., morphisms from $\mathcal X$ to $\mathcal Y$ in $\mathcal H(k)$.
\item Given a smooth $k$-scheme $X$, we abuse notation and use the same symbol $X$ for the motivic space associated to the motivic localization of the image of $X$ under the Yoneda embedding from $\operatorname{Sm}_k$ to presheaves of spaces on $\Sm_k$. 
\item We write $\GL_r$ for the $r \times r$ general linear group as a smooth $k$-scheme, and $\Sp_r$ for the group of $r \times r$ symplectic matrices.
\item\label{def:stiefel} We write $\ST{i}{j}$ and $\STsp{2i}{2j}$ ordinary and symplectic Stiefel varieties $\GL_j/\GL_{j-i}$ and $\Sp_{2j}/\Sp_{2j-2i}$, respectively. 
\item\label{def:Gras} Let $j>i\geq 0$ be integers. We write $\Gr{i}{j}= \GL_{j}/(\GL_{i} \times \GL_{j-i})$ for the the Grassmannian of $i$-planes in $\mathbb A^j$. Similarly, we write $\Grsp{2i}{2j} = \Sp_{2j}/(\Sp_{2i}\times \Sp_{2j-2i})$ for the symplectic Grassmannian.
\item We assume familiarity with affine representability results for vector bundles and symplectic vector bundles (see \cite[\S 8.1]{Morel}, \cite[Theorem 1]{AHW1}, and \cite[Theorem 1]{AHW2} for details). In particular, for $X$ a smooth affine variety over $k$, let $\operatorname{Vect}_r(X)$ denote the set of isomorphism classes of algebraic vector bundles of rank $r$ over the variety $X$. There is a natural bijection
$$\operatorname{Vect}_r(X) \cong [X,\BGL_r],$$
where we refer the reader to \cite[Page 1010]{AHW2} for the definition the classifying space functor $\operatorname B(-)$. Similarly, if $\operatorname{Vect}_{2r}^{\operatorname{Sp}}(X)$ denotes the set of isomorphism classes of symplectic rank $2r$ vector bundles over $X$, there is a natural bijection
$$\operatorname{Vect}_r^{\operatorname{Sp}}(X) \cong [X,\BSp_{r}].$$
\item We make extensive use of Moore--Postnikov theory in $\mathbb A^1$-homotopy theory. For the details of this obstruction theory, we refer the reader to \cite[Appendix~B]{Morel} and \cite[\S 6.1]{AF-splitting}.
\item Given a finite product of motivic spaces or stricly $\A^1$-invariant sheaves, we write $\operatorname{pr}_i$ for projection onto the $i$-th factor.

\end{enumerate}

\subsection{Acknowledgements} The second-named author was supported by an NSF Postdoctoral Research Fellowship, Award No. 2202914. The fourth-named author was partially funded by the Deutsche Forschungsgemeinschaft (DFG, German Research Foundation) - Project numbers 461453992; 544731044.

\section{Homotopy sheaves of Stiefel varieties}\label{sec:homotopy_sheaves}

In this section, we study certain homotopy sheaves of Stiefel varieties. In particular, we investigate $\pia_i\ST{n-j}{n}$ for $n>j>1$ and $i=j,j+1$ (recall notation from \Cref{subsec:prelims}, \Cref{def:stiefel}). These homotopy sheaves govern obstruction-theoretic problems involving efficient generation of projective modules. We begin in \Cref{sec:fiber} by establishing key fiber sequences. \Cref{subsec:stablerange} is a stepping stone, providing a stable range and vanishing results for homotopy sheaves of Stiefel varieties; we also note some results on weak cellularity classes of Stiefel varieties, which is a stronger notion than connectivity. In \Cref{subsec:actual_computations} we provide our main results on homotopy sheaves. \Cref{sebsec:seq-add} provides some additional technical results that are not used in our obstruction theory but that may be of independent interest.
\subsection{Fiber sequences involving Stiefel varieties}\label{sec:fiber}
We construct some fiber sequences that we will use both for computing homotopy sheaves in \Cref{subsec:stablerange,subsec:actual_computations}, and for our study of efficient generation in \Cref{total-obstructions}. We refer the reader to \Cref{def:stiefel} and \Cref{def:Gras} for notation used throughout this section.
\begin{proposition}\label{lem:varyN} Let $n>j\geq 1$. There are fiber sequences
\begin{align}
& \ST{n-j}{n} \to \BGL_{j} \to \BGL_n, \label[sequence]{eq:vary_rank}\\
&\ST{n-j}{n} \to  \Gr{j}{n} \to\BGL_j,   \label[sequence]{prop:fib_stiefel_pf}  \\
&\Omega (\A^{n+1}\setminus 0) \to \ST{n-j}{n} \to\ST{n-j+1}{n+1},\text{ and} \label[sequence]{eq:varyN1}\\
&\A^{j+1}\setminus 0 \to \ST{n-j}{n} \to \ST{n-j-1}{n}.\label{fiber2}
\end{align}
\end{proposition} 
\begin{proof}
\Cref{eq:vary_rank} follows by combining \cite[Theorem 2.2.5]{AHW2} with \cite[Lemma 2.4.1]{AHW2}, 
which we also appeal to for the other items in combination with additional arguments.
For \Cref{prop:fib_stiefel_pf}, apply \cite[Proposition 3.1]{ABHWhitehead} to the commutative square:
	\[
	\begin{tikzcd}[row sep=1.5em, column sep = 2em]
	\BGL_j \times \BGL_{n-j}\ar[r,"\mathrm{pr}_1"]\ar[d,"\oplus"] & \BGL_{j} \ar[d]\\
	\BGL_n \ar[r] & *, 
	\end{tikzcd}
	\]
	where the left-hand vertical map classifies direct sum of vector bundles: we take the fiber of the induced map of horizontal fibers.
 \Cref{eq:varyN1} follows from applying simplicial loops to the base in the sequence \[\GL_n/\GL_j \to \GL_{n+1}/\GL_j\to \GL_{n+1}/\GL_n \simeq \A^{n+1} \setminus 0,\] which is a fiber sequence by  \cite[Theorem 2.2.5 and Lemma 2.4.1]{AHW2} again. 
For the last sequence, the argument is analogous \Cref{prop:fib_stiefel_pf}: the inclusions $\GL_j \subset \GL_{j+1} \subset \GL_n$ yields a commutative square of maps of classifying spaces
\[
\begin{tikzcd}[row sep=1.5em, column sep=1.5em]
	\BGL_j \ar[r]\ar[d]& \BGL_{j+1}\ar[d]\\
	\BGL_n \ar[r,"\operatorname{id}"]& \BGL_{n},
\end{tikzcd}
\]	
and we apply \cite[Proposition 3.1]{ABHWhitehead}. 
\end{proof}
We also make note of the following commutative diagram, which we will use to compare obstructions to efficient generation to better-understood obstructions.
\begin{proposition}\label{prop:big_diagram-sp} For $n>j \geq 1$, there is a commutative diagram:
\begin{equation}\label[diagram]{diag:big_diagram}
\begin{tikzcd}[column sep=1.5em, row sep=1.5em]
\ST{j}{n} \ar[r]&\Gr{j}{n} \ar[r] & \BGL_j \\
(\A^{n-j+1}\setminus 0)\times \Omega (\A^{n+1}\setminus 0) \ar[r]\ar[u]\ar[d,"\operatorname{pr}_1"] & \Gr{j}{n} \ar[r]\ar[u]\ar[d]& \Gr{j}{n+1}\ar[u] \ar[d]\\
\A^{n-j+1}\setminus 0 \ar[r]  \ar[r] & \BGL_{n-j} \ar[r] &\BGL_{n-j+1}.
\end{tikzcd}
\end{equation}
 Moreover, the horizontal sequences are all fiber sequences.

\end{proposition}
\begin{proof}
The bottom row is a fiber sequence since $\GL_{n-j+1}/\GL_{n-j} \cong \A^{n-j+1} \setminus 0.$
For the second row, consider the iterated fiber of the commutative diagram:

\begin{equation}\label[diagram]{diag:big} \begin{tikzcd}[column sep=1.5em, row sep=1.5em]
& \A^{n-j+1}\setminus 0\ar[d] \ar[r] & \A^{n+1}\setminus 0\ar[d]\\
\Gr{j}{n} \ar[r]\ar[d]& \BGL_j \times \BGL_{n-j}  \ar[r]\ar[d]& \BGL_{n} \ar[d]\\ 
\Gr{j}{n+1}\ar[r]& \BGL_j \times \BGL_{n-j+1} \ar[r] & \BGL_{n+1}.
\end{tikzcd}
\end{equation}
By \cite[Corollary 5.43]{Morel}, the top horizontal map $\A^{n-j+1}\setminus 0 \to \A^{n+1} \setminus 0$ is nullhomotopic, so the fiber is  $\A^{n-j+1 }\setminus 0 \times \Omega \A^{n+1}\setminus 0.$ We conclude by \cite[Proposition 3.1]{ABHWhitehead}.
\end{proof}

We also note that there are various analogous fiber sequences with general linear groups replaced by symplectic groups. 
The proofs are completely analogous to the versions for general linear groups.
\begin{proposition}\label{prop:symplectic_fiber_seqs} Let $n> j\geq 1$ be given. There are fiber sequences:
\begin{align}& \STsp{2n-2j}{2n} \to \BSp_{2j} \to \BSp_{2n}, \label[sequence]{seq:rank_symp}\\
&\STsp{2n-2j}{2n} \to \Grsp{2j}{2n}\to \BSp_{2j},\label[sequence]{seq:stiefel} \\
&\Omega (\A^{2n+2}\setminus 0) \to \STsp{2n-2j}{2n} \to\STsp{2n-2j+2}{2n+2},\text{ and} \label[sequence]{eq:varyN}\\
& \A^{2j+2}\setminus 0 \to \STsp{2n-2j}{2n} \to \STsp{2n-2j-2}{2n}.\label[sequence]{seq:change_quotient}\
\end{align}
\end{proposition}

Moreover, we have a symplectic version of \Cref{diag:big_diagram}, proved via a familiar argument.

\begin{proposition}\label{prop:symplectic_big_diagram}
 For $n>j \geq 1$, there is a commutative diagram:
\begin{equation}
\begin{tikzcd}[column sep=1.5em, row sep=1.1em]
\STsp{2j}{2n} \ar[r]&\Grsp{2j}{2n}\ar[r] & \BSp_{2j} \\
(\A^{2(n-j+1)}\setminus 0)\times (\Omega \A^{2(n+1)}\setminus 0) \ar[r]\ar[u]\ar[d,"p_1"] & \Grsp{2j}{2n}\ar[r]\ar[u]\ar[d]& \Grsp{2j}{2n+2}\ar[u] \ar[d]\\
\A^{2(n-j+1)}\setminus 0 \ar[r]  \ar[r] & \BSp_{2n-2j} \ar[r] &\BSp_{2n-2j+2} 
\end{tikzcd}
\end{equation}
where each horizontal row is a fiber sequence.
\end{proposition}

\subsection{Stabilization via weak-cellular/connectivity estimates}\label{subsec:stablerange}
Note that
\[\ST{1}{j+1} =\GL_{j+1}/\GL_j \simeq \A^{j+1}\setminus 0.\] For $n>j+1$, the Stiefel variety $\ST{n-j}{n}=\GL_n/\GL_{j} $ does not admit such a simple description. However, computations of $\pia_{j}\ST{n-j}{n}$ and $\pia_{j+1} \ST{n-j}{n}$ stabilize as $n$ grows. 
\begin{lemma}\label{lem:cellularity1} For $n> j \geq 1$, the natural map $\ST{n-j}{n} \to \ST{n-j+1}{n+1}$ is a universal $S^{2n, n+1}$-equivalence, in the sense of \cite[Section 3]{ABH}.
\end{lemma}
\begin{proof}
Note that $\mathbb A^{n+1}\setminus 0 \simeq S^{2n+1,n+1}.$ Therefore by \cite[Proposition 4.2.1]{ABH}, $\Omega\mathbb A^{n+1}\setminus 0$ is weakly $S^{2n,n+1}$-cellular. By \cite[Proposition 3.1.23]{ABH}, the map $\GL_n/\GL_j \to \GL_{n+1}/\GL_j$ is a universal $S^{2n, n+1}$-equivalence.\end{proof}
\begin{corollary}\label{cor:stable-range} For $i \geq 0$ and $n>j\geq1$ the map $\pia_{i}\ST{n-j}{n} \to \pia_i\ST{n+1-j}{n+1}$ is:
\begin{enumerate}
\item An isomorphism if $i\leq j$ and $n\geq j+2$, or $i\leq j+1$ and $n\geq j+3$.
\item An epimorphism if $i=j$ and $n=j+1$, or $i=j+1$ and $n=j+2$.
\end{enumerate}
\end{corollary}
\begin{proof} This follows from \cite[Lemma 3.1.19]{ABH}. \end{proof}
Inductively, starting from the fact that $\A^{j+1} \setminus 0$ is $(j-1)$-connected, we deduce:
\begin{corollary}\label{cor:connectivity} For each $n>j\geq 1$, the Stiefel variety $\ST{n-j}{n}$ is $(j-1)$-connected.
\end{corollary}

We note some results for symplectic Stiefel varieties, proved by completely analogous arguments to the corresponding ones above.
\begin{lemma}\label{lem:cellularity2} For $n> j \geq 1$, the map $\STsp{2n-2j}{2n} \to \STsp{2n+2-2j}{2n+2}$ is an $S^{4n+2,2n+2}$-equivalence.
\end{lemma}
\begin{corollary}\label{cor:stable-range-first-homotopy} For $i\geq 0$ and $n> j \geq 1$, the map 
$\pia_{i}\STsp{2n-2j}{2n} \to \pia_{i}\STsp{2n+2-j}{2n+2}$ is:
\begin{enumerate}
\item An isomorphism if $i\leq 2j+1$, or if $i=2j+2$ and $n>j+1$.
\item An epimorphism if $i=2j+2$ and $n=j+1$.
\end{enumerate}
\end{corollary}
\begin{corollary}\label{cor:connectivity-sp} For each $n>j\geq 1$, the symplectic Stiefel variety $\STsp{2n-2j}{2n}$ is $2j$-connected.
\end{corollary}
\subsection{Computing homotopy sheaves of Stiefel varieties}\label{subsec:actual_computations}
We now directly study $\pia_{j}\ST{n-j}{n}$ and \\$\pia_{j+1}\ST{n-j}{n}$ for $n>j+1$. Note that $\pia_{j}\ST{1}{j+1}\cong \KMW_{j+1}$ by \cite[6.40]{Morel}. 
The homotopy sheaf \[\pia_{j+1}\ST{1}{j+1}\cong \pia_{j+1}(\A^{j+1}\setminus 0)\] 
is not computed, but it has been extensively studied and so is often amenable to cohomology computations. 
In light of \Cref{cor:stable-range}, the problem for $n>j+1$ reduces to computing $\pia_j \ST{2}{j+2}$ (\Cref{pirr2}),  $\pia_{j+1}\ST{2}{j+2}$ (\Cref{pir1GLr2}), and
$\pia_{j+1} \ST{3}{j+3}$ (\Cref{pir1GLr3}).
To approach these groups, we use certain fiber sequences from the previous section, which we reintroduce below to establish notation.
\begin{definition}\label{diag:abnotation}
We let $a_{n-j}^j$ and $b_{n-j}^j$ denote the morphisms in the fiber sequence of \Cref{fiber2}, so that we have a fiber sequence:
\[\A^{j+1}\setminus 0 \xrightarrow{a_{n-j}^j}\ST{n-j}{n}  \xrightarrow{b_{n-j}^j} \ST{n-j-1}{n}\]
\end{definition} 
The fiber sequence in \Cref{diag:abnotation} induces a long exact sequence on homotopy sheaves:
\begin{equation}\label[sequence]{LES-key}\begin{tikzcd}[row sep=2.5em, column sep = 4em] 
\pia_{j+2}(\A^{j+1}\setminus 0) \ar[r,"\pi_{j+2}a_{n-j}^j" below] 
& \pia_{j+2}\ST{n-j}{n} \ar[r,"{\pi_{j+2}b_{n-j}^j}"] 
& \pia_{j+2}\ST{n-j-1}{n} \ar[dll,"\partial_{n-j}^j" near start]\\
\pia_{j+1}(\A^{j+1}\setminus 0) \ar[r,"\pi_{j+1}a_{n-j}^j" below] 
& \pia_{j+1}\ST{n-j}{n} \ar[r,"{\pi_{j+1}b_{n-j}^j}"] 
& \pia_{j+1}\ST{n-j-1}{n}\ar[dll,"d_{n-j}^j" near start]\\
\pia_{j}(\A^{j+1}\setminus 0)\ar[r,"{\pi_ja_{n-j}^j}" below] 
& \pia_j\ST{n-j}{n} \ar[r,"{\pi_{j}b_{n-j}^j}"] &0.
\end{tikzcd}
\end{equation}
Note also that we have compatibility between the instances of \Cref{LES-key} as $n$ varies: the commutative diagram
 \begin{equation}\label[diagram]{diag:abnotation-compare}\begin{tikzcd}[column sep = 3.5em]\A^{j+1}\setminus 0 \ar[r,"{a_{n-j}^j}"] \ar[d,"\simeq"]& \ST{n-j}{n} \ar[r,"{b_{n-j}^j}"] \ar[d]&\ST{n-j-1}{n}\ar[d]\\
\A^{j+1}\setminus 0  \ar[r,"{a_{n-j+1}^j}"] & \ST{n+1-j}{n+1} \ar[r,"{b_{n-j+1}^j}"]&\ST{n-j}{n+1} \\
\end{tikzcd}\end{equation} induces a map of long exact sequences. 
\begin{lemma}\label{pirr2} Let $j \geq 1$. Then 
\[\pia_j\ST{2}{j+2} \simeq \begin{cases} \KMW_{j+1} & j \text{ even }\\
\KM_{j+1} & j \text{ odd.}\end{cases}\]
Moreover, the morphism $\pia_{j}a_2^j\: \pia_{j}(\A^{j+1}\setminus 0 )\to\pia_j \ST{2}{j+2}$ is an isomorphism for $j$ even and the quotient of  $\KMW_{j+1}$ by $\eta$ for $j$ odd.
\end{lemma} 
\begin{proof} Consider a portion of \Cref{LES-key} in the case $n=j+2$:

\begin{equation}\label[empty]{ES2} \pia_{j+1}\ST{1}{j+2} \xrightarrow{d_2^j} \pia_{j} (\A^{j+1} \setminus 0) \xrightarrow{\pi_{j}a_2^j}
\pia_j\ST{2}{j+2}\to 0.
\end{equation}
Note that $\pia_{j+1}\ST{1}{j+2} \simeq \KMW_{j+2},$ so we have an exact sequence
\begin{equation}\label[diagram]{ES3} \KMW_{j+2} \xrightarrow{d_2^j} \KMW_{j+1}\to 
\pia_j\GL_{j+2} / \GL_j\to 0
\end{equation}
and it remains to compute $d_2^j$. By \cite[Lemma 3.5]{AF14}, we  can identify $d_2^j$ up to automorphism of $\KMW_{j+2}$ as follows:
\begin{equation}\label[empty]{defd}d_2^j=\begin{cases} 0 & j \text{ even }\\
\eta & j\text{ odd.}\end{cases}\end{equation}
Therefore $\pia_{j}\ST{2}{j+2} \cong \operatorname{coker}(d_2^j)$, which is  $\KMW_{j+1}$ for $j$ even and
$\KM_{j+1} $ for $j$ odd.
\end{proof}
By \Cref{cor:stable-range}, we deduce:
\begin{corollary}\label{cor:general-first-sheaf-compuation} Let $n\geq j+2\geq 3$ be arbitrary. Then 
\[\pia_j\ST{n-j}{n} \cong \begin{cases} \KMW_{j+1} & j \text{ even }\\
\KM_{j+1} & j \text{ odd.}\end{cases}\]
Moreover, the map $a_{n-j}^j\: \A^{j+1}\setminus 0 \to \ST{n-j}{j}$ 
from \Cref{diag:abnotation} induces an epimorphism
\[\pia_{j}a^j_{n-j}\: \KMW_{j+1} \to \pia_{j} \ST{n-j}{n}.\]
 When $j$ is even, $\pia_ja^j_{n-j}$ is an isomorphism;
when $j$ is odd, $\pia_ja^j_{n-j}$ is the quotient quotient of $\KMW_{j+1}$ by $\eta$.
\end{corollary}
The computation of $\pia_{j}a_{n-j}^j$ and $d_{n-j}^j$ allows us to deduce some results about terms involving $(j+1)$-st homotopy groups in \Cref{LES-key}.
\begin{proposition}\label{prop:necessary-computation}
Let $n \geq j+2$. For $j\geq 0$ is even, we have an exact sequence
\[ \pia_{j+1}(\A^{j+1} \setminus 0) \xrightarrow{\pia_{j+1} a_{n-j}^j} \pia_{j+1}\ST{n-j}{n} \xrightarrow{\pia_{j+1}b_{n-j}^j}\pia_{j+1} \ST{n-j-1}{n} \to 0\]
For $j \geq 1$ odd, we have an exact sequence
\[ \pia_{j+1}(\A^{j+1} \setminus 0) \xrightarrow{\pia_{j+1} a_{n-j}^j} \pia_{j+1}\ST{n-j}{n} \xrightarrow{}2\KM_{j+2}\to 0,\]
where the identification of the last term is via $2\KM_{j+2} \cong \operatorname{Im}(b_{n-j}^j) \hookrightarrow \pi_{j+1}\ST{n-j-1}{j}$.
\end{proposition} 
\begin{proof} Consider \Cref{LES-key}. By \Cref{cor:general-first-sheaf-compuation}, 
for $j$ even $\pia_{j}a_{n-j}^j$ is an isomorphism and $d_{n-j}^j$ factors through zero. 
For $j$ odd, $\pia_{j+1}\ST{n-(j+1)}{n} \cong \KMW_{j+2}$ and $d_{n-j}^j \cong \eta$ so that \[\operatorname{Ker}(d_{n-j}^j) =\operatorname{Im}(\pia_{j+1}b_{n-j}^j)\cong 2\KM_{j+2}.\]
\end{proof}
We now consider homotopy sheaves of symplectic Stiefel varieties.
Since \[\pia_{2j+1}\STsp{2}{2j+2} \cong \pia_{2j+1}(\mathbb A^{2j+2}\setminus 0)\cong \KMW_{2j+2},\] we apply \Cref{cor:stable-range-first-homotopy} to conclude:
\begin{corollary}\label{cor:first-symp-homotopy} For any $n> j \geq 1$, $\pia_{i}\STsp{2n-2j}{2n}=0$ for $i<2j+1$ and $\pia_{2j+1}\STsp{2n-2j}{2n}\cong \KMW_{2j+2}.$ \end{corollary}
We now consider the second nontrivial homotopy sheaf.
\begin{proposition}\label{prop:symplectic-2jp4} For each $n>j+1$ and $j\geq 1$, there is an exact sequence 
\[\KMW_{2j+4} \to \pia_{2j+2}(\A^{2j+2}\setminus 0) \to \pia_{2j+2}\STsp{2n-2j}{2n} \to 0.\]
\end{proposition}
\begin{proof} This follows from the $\A^1$-fiber sequence $\A^{2j+2}\setminus 0 \to \STsp{2n}{2n-2j} \to \STsp{2n}{2n-2j-2}$ from \Cref{prop:symplectic_fiber_seqs} and its long exact sequence of homotopy sheaves, using  \Cref{lem:cellularity2} and \Cref{cor:first-symp-homotopy} to identify terms.
\end{proof}

\subsection{Additional exact sequences involving $\pia_{j+1}\ST{n-j}{n}$}\label{sebsec:seq-add} In this section, we give some more detailed results on exact sequences involving $\pia_{j+1}\ST{n-j}{n}$, strengthening \Cref{prop:necessary-computation}.
\begin{lemma}\label{pir1GLr2} 
When $j$ is even, \[\pia_{j+1}\ST{2}{j+2} \simeq \KMW_{j+2} \oplus \pia_{j+1}(\A^{j+1} \setminus  0).\] When $j$ is odd, there is an exact sequence
\begin{equation}\label[diagram]{2KMWend}\pia_{j+2}(\A^{j+2}\setminus 0 )\xrightarrow{\partial_2^j} \pia_{j+1} (\A^{j+1}\setminus 0)\xrightarrow{\pi_{j+1}a_{2}^j} \pia_{j+1}\ST{2}{j+2} \to 2\KM_{j+2}\to 0.\end{equation}
\end{lemma}
\begin{proof} 
First, suppose that $j$ is even. The map $b_2^j\: \ST{2}{j+2} \to \ST{1}{j+2}$ 
has a section given by
the natural map $\iota\: \STsp{2}{j+2}\to \ST{2}{j+2}$. Indeed, by the discussion before \cite[Lemma 4.2.1]{AF-splitting}, the composite $b_{2}^j \circ \iota$ is an isomorphism. Thus \Cref{LES-key} with $n=j+2$ breaks into split short exact sequences, and we have a split short exact sequence
\[ 0 \to \pia_{j+1}(\A^{j+1}\setminus 0) \to \pia_{j+1}\ST{2}{j+2} \to \KMW_{j+2} \to 0.\] 
When $j$ is odd, we cannot appeal to symplectic groups for a splitting. Consider again \Cref{LES-key} with $n=j+2$:
\[ \pia_{j+1}\ST{2}{j+2} \xrightarrow{\pi_{j+1}b_{2}^j} \pia_{j+1}\ST{1}{j+2} \xrightarrow{d_2^j} \pia_j(\A^{j+1} \setminus 0),\]
We have already shown that $d_2^j$ can be identified with $\eta$ when $j$ is odd. Since $\operatorname{ker}\{\eta\} = 2\KM_{j+2}$, we obtain the stated exact sequence.
\end{proof}

\begin{lemma}\label{pir1GLr3}  Let $j\geq1$. If $j$ is even, we have an exact sequence
\[2\KM_{j+3} \to \pia_{j+1}(\A^{j+1}\setminus 0) \to \pia_{j+1}\ST{3}{j+3}\to \KM_{j+2} \to 0.\]

If $j$ is odd, we have an exact sequence
\[\KMW_{j+3} \oplus \pia_{j+2}(\A^{j+2}\setminus 0)\to \pia_{j+1}(\A^{j+1}\setminus 0) \to 
\pia_{j+1}\ST{3}{j+3}\to 
2\KM_{j+2}\to 0.\]
\end{lemma}
\begin{proof}
We take the long exact sequence associated with \Cref{fiber2} for $n=j+3$:

\begin{equation}\label[sequence]{Lemma6_9}\begin{tikzcd}[row sep=1.1em]
\pia_{j+2}\ST{2}{j+3} \ar[r,"\partial_3^j"] & \pia_{j+1}(\A^{j+1}\setminus 0) \ar[dl]\\
\pia_{j+1}\ST{3}{j+3}\ar[r]&
\pia_{j+1}\ST{2}{j+3}\ar[r,"{d_3^j}"] & \pia_j(\A^{j+1}\setminus 0),\end{tikzcd}\end{equation}
where $d_3^j$ and $\partial_3^j$ are the as in \Cref{LES-key}. 

First, consider the kernel of the map $d_3^j$. We have a commutative diagram:
\begin{equation}\label[sequence]{Lemma6_9}\begin{tikzcd}
\pia_{j+1}\ST{3}{j+3} \ar[r]&
\pia_{j+1} \ST{2}{j+3} \ar[r,"{d_3^j}"] & \pia_j(\A^{j+1}\setminus 0)\\
 \pia_{j+1}\ST{2}{j+2} \ar[r]\ar[u]& \pia_{j+1}\ST{1}{j+2}\ar[r,"d_2^j"]\ar[u]&\pia_j(\A^{j+1} \setminus 0) \ar[u,"\simeq"]\end{tikzcd}\end{equation}
where $d_2^j$ is again from \Cref{LES-key}.
By \Cref{cor:stable-range}, the middle vertical map is surjective. By \Cref{defd}, we see that $d_3^j$ is zero when $j$ is even. By \Cref{cor:general-first-sheaf-compuation} we get a an exact sequence $\pia_{j+1}\ST{3}{j+3} \to \KM_{j+2} \to 0.$ For $j$ odd, the middle vertical map is an isomorphism. By \Cref{defd}, $d_2^j$ is multiplication by $\eta$ and we obtain an exact sequence 
\[\pia_{j+1}\ST{3}{j+3} \to 
2\KM_{j+2}\to 0.\]

We now study $\partial_3^j$. For $j$ even, 
 consider the diagram
\begin{equation}\label[diagram]{twoseq}
\begin{tikzcd}
\pia_{j+2}(\A^{j+2} \setminus 0)\ar[d,"\pi_{j+2}a_2^{j+1}"]\\
 \pia_{j+2}\ST{2}{j+3} \ar[r,"\partial_{3}^j"] \ar[d]& \pia_{j+1}(\A^{j+1}\setminus 0) \\
2\KM_{j+3}\ar[d]\\
0,
\end{tikzcd}
\end{equation}
with the vertical exact sequence as in \Cref{2KMWend}. We claim that the composite $\partial_3^j \circ \pi_{j+2}a_2^{j+1}=0$, which completes the lemma in the case $j$ even. To prove the claim, recall that 
\begin{equation}\label[empty]{eq:Sp_identification}\STsp{2}{j+2} = \Sp_{j+2}/\Sp_{j} \simeq \GL_{j+2}/\GL_{j+1}\simeq \A^{j+2}\setminus 0.\end{equation}  

Under this identification, we see that we have a commutative diagram:
\begin{equation}\label[diagram]{blh}
\begin{tikzcd}
\STsp{2}{j+2} \ar[r,"\simeq"]\ar[d,"c"]&\A^{j+2} \setminus 0\ar[d,"a_2^{j+1}"]\\
\ST{3}{j+3}\ar[r] &\ST{2}{j+3},
\end{tikzcd}
\end{equation}
where $c$ is induced by the natural map $\Sp_{j+2} \to \GL_{j+2} \to \GL_{j+3}$. The induced map on fibers gives a diagram
\begin{equation}\label[diagram]{blh2}
\begin{tikzcd}
* \ar[d]\ar[r]&\STsp{j+2}{2} \ar[r,"\simeq"]\ar[d,"c"]&\A^{j+2} \setminus 0\ar[d,"a_2^{j+1}"]\\
\A^{j+1} \setminus 0 \ar[r]&\ST{3}{j+3}\ar[r] &\ST{2}{j+3}
\end{tikzcd}
\end{equation}
The induced commutative diagram on long exact sequences of homotopy sheaves includes a portion:

\begin{equation}\label[diagram]{blh2}
\begin{tikzcd}
\piA_{j+2}(\A^{j+2} \setminus 0)\ar[d,"\pi_{j+2}a_2^{j+1}"]\ar[r] &\piA_{j+1}(*)\simeq 0 \ar[d]\\
\piA_{j+2}\ST{2}{j+3}\ar[r,"\partial_3^{j}"]& \piA_{j+1}(\A^{j+1}\setminus 0),
\end{tikzcd}
\end{equation}
show that $\partial_3^j\circ \pi_{j+2}a_2^{j+1}$ factors through zero.

For  $j$ odd,  \Cref{pir1GLr2} tells us that \[\pia_{j+2}\ST{2}{j+3}\simeq \KMW_{j+3} \oplus \pia_{j+2}(\A^{j+2}\setminus 0),\] which gives the result.
\end{proof}

\begin{remark}
It is possible to identify some of the morphisms in the exact sequences above with differentials in the linear spectral sequence considered in \cite[Section 2.1]{AF4} or with differentials in the symplectic spectral sequence considered in \cite[Section 2.2]{AF4}. We summarize these identifications here:
\begin{itemize}
\item For $j$ odd, the morphism $\partial_2^j$ as in \Cref{2KMWend} is the differential $d_{j+2,0}^{1}$ in the linear spectral sequence.
\item For $j$ even, the morphism \[2\KM_{j+3} \rightarrow \pia_{j+1}(\mathbb{A}^{j+1}\setminus 0)\] in \Cref{pir1GLr3} is the composite \[2\KM_{j+3} \xrightarrow{d_{j+3,-1}^{2}} E_{j+1,0}^{2} \hookrightarrow \pia_{j+1}(\mathbb{A}^{j+1}\setminus 0),\] where $d_{j+3,-1}^{2}$ is the indicated differential in the linear spectral sequence.
\item For $j$ odd, the morphism \[\pia_{j+2}(\mathbb{A}^{j+2}\setminus 0) \rightarrow \pia_{j+1}(\mathbb{A}^{j+1}\setminus 0)\] in \Cref{pir1GLr3} is the differential $d_{j+2,0}^{1}$ in the linear spectral sequence.
\item If $j$ is odd, the morphism \[\KMW_{j+3} \rightarrow \pia_{j+1}(\mathbb{A}^{j+1}\setminus 0)\] in \Cref{pir1GLr3} is the differential $d^{1}_{\frac{j+3}{2},\frac{j+3}{2}-1}$ in the symplectic spectral sequence.
\item The morphism \[\KMW_{2j+4} \rightarrow \pia_{2j+2}(\mathbb{A}^{2j+2} \setminus 0)\] in \Cref{prop:symplectic-2jp4} is the differential $d^{1}_{j+2,j+1}$ in the symplectic spectral sequence.
\end{itemize}
\end{remark}

\section{Motivic obstruction theory and efficient generation}\label{total-obstructions}

Let $k$ be an algebraically closed field and $X=\spec R$ a smooth affine $k$-variety of dimension $d$. Murthy's celebrated work implies that the $d$-th Segre class of a projective module $M$ of rank $r$ over  $X$ is the only obstruction to generating $M$ by $r+d-1$ elements \cite{Mu}. If we additionally assume that $k$ has characteristic zero, Murthy's splitting conjecture in characteristic zero implies the following:
\begin{proposition}\label{char-zero-segre-vanish-rd2}Let $k$ be an algebraically closed field of characteristic zero. Let $X=\spec R$ be a smooth affine variety of dimension $d\geq 2$ over $k$. Let $M$ be a projective module of rank $r \geq 1$ over $R.$ Then $M$ can be generated by $r+d-2$ elements if and only if the $d$-th and $(d-1)$-st Segre classes of $M$ vanish.
\end{proposition} 
We refer the reader to \Cref{def:seg} for the background on the Segre class.
\begin{proof}[Proof of \Cref{char-zero-segre-vanish-rd2}]
$M$ can be generated by $r+d-2$ elements if and only if there is a rank $d-2$ projective module $Q$ over $R$ such that $M \oplus Q$ is free. If such an $Q$ exists, then its $d$-th and $(d-1)$-st Chern classes vanish. These are the $d$-th and $(d-1)$-st Segre class of $M$ by definition.

Conversely, suppose that the $d$-th and $(d-1)$-st Segre class of $M$ vanishes and that $Q$ is a rank $d$ projective module such that $M \oplus Q$ is free (such an $Q$ exists by the Forster--Swan theorem). By hypothesis, the $d$-th and $(d-1)$-st Chern classes of $Q$ vanish. By Murthy's conjecture in characteristic zero \cite[Theorem 7.1.1]{ABH}, $Q\simeq Q' \oplus R^2$ for some projective module $Q'$ of rank $d-2$. We need a cancellation result to conclude that $M\oplus Q'$ is free. For $r\geq 3$, this is Bass cancellation. For $r=2$, we appeal to Suslin's celebrated work \cite{suslin77a}. In the case $r=1$, we use Suslin's Cancellation conjecture as resolved by Fasel  \cite[Theorem 2]{fasel2021suslin}.
\end{proof}
The main project of this section is to explore an obstruction-theoretic approach to reducing the number of generators of a projective module over smooth $k$-algebras when $k$ is not necessarily algebraically closed and not necessarily of characteristic zero. In \Cref{subsec:absolute} we explain the relevant obstruction theory. In \Cref{subsec:classical-invariants}, we review some classical vector bundle invariants, including Segre and Euler classes. In \Cref{subsec:generation-line-bundles}, we study efficient generation for line bundles. In \Cref{subsec:first-obstr-classical} and \Cref{further-reduction}, we identify certain key obstructions with Segre or Euler classes and give conditions for rank $r$ vector bundles on a smooth affine variety of Nisnevich cohomological dimension at most $d$ to be $(r+d-1)$- or $(r+d-2)$-generated.

\subsection{Setting up obstruction theory for $\Gr{r}{n}\to \BGL_r$}\label{subsec:absolute}
Let $X=\spec R$ be a smooth affine variety of $\A^1$-cohomological dimension at most $d$ over a perfect field $k$ and let $ n\geq r+2$. We consider the Moore-Postnikov factorization of the morphism $\Gr{r}{n} \to \BGL_r$ representing the tautological rank $r$ bundle on $\Gr{r}{n}$. 
We follow \cite[Section 6.1]{AF-splitting} for the obstruction-theoretic set-up. By \Cref{lem:varyN}, this morphism fits into a fiber sequence
\[
\ST{r}{n} \rightarrow \Gr{r}{n} \rightarrow \BGL_{r},
\] so the obstruction groups will be cohomology groups with coefficients in homotopy sheaves of $\ST{r}{n}$.

Let $M\: X \rightarrow \BGL_{r}$ be a morphism representing a finitely generated projective $R$-module of rank $r$. Assuming one can lift $M$ to the $(i-1)$-st stage of the Moore--Postnikov factorization for the morphism $\Gr{r}{n} \to \BGL_r$, the obstruction to lifting $M$ to the $i$-th stage can be identified with an element

\begin{equation}\label{eq:obstructions-reducing}
o_{i,n,r}(M) \in H^{i+1} (X, \pi_{i} \ST{r}{n}(\det M)),
\end{equation}
which is only well-defined up to the choice of a lift of $M$ to the $(i-1)$-st stage. If $i \geq d$, then this obstruction vanishes automatically. 
By \Cref{lem:cellularity1}, the first potentially non-trivial obstruction to lifting $M$ to a map $\tilde{M}\:X \rightarrow \Gr{r}{n}$ is the element $o_{n-r,n,r}(M)$ in an $(n-r+1)$-st Nisnevich cohomology group of $X$.
Taking $n=r+d$, we find that all obstructions are identically zero. This completes our proof of a homotopy Forster--Swan theorem:

\begin{theorem}\label{htpy:Fo-Sw}
Let $X = \spec R$ be a smooth affine variety of $\A^1$-cohomological dimension at most $d \geq 2$ over a perfect field $k$. Any finitely generated projective $R$-module of rank $r$ can be generated by $n=r+d$ elements.
\end{theorem}
The framework above gives some immediate answers for when a projective module can be generated by $r+d-1$ or $r+d-2$ elements:

\begin{lemma}
Let $X = \spec R$ be a smooth affine variety of $\A^1$-cohomological dimension at most $d \geq 3$ over a perfect field $k$ and let $M\: X \to \BGL_r$ represent a finitely generated projective $R$-module of rank $r$. Then $M$ is generated by $n=r+d-1$ elements if and only if $o_{d-1,n,r}(M) = 0 \in  H^{d} (X, \pia_{d-1} \ST{r}{n}(\det M))$. 
\end{lemma}

\begin{lemma}
Let $X = \spec R$ be a smooth affine variety of dimension $d \geq 4$ over a perfect field $k$. Suppose that, for any line bundle $\mathcal L$ on $X$,  $H^{d-1} (X, \pia_{d-2} \ST{r}{r+d-2}(\mathcal L))$ and $H^{d} (X, \pia_{d-1} \ST{r}{r+d-2}(\mathcal L))$ are trivial. Then all finitely generated projective $R$-modules of rank $r$ can be generated by $r+d-2$ elements. 
\end{lemma}
In the next section, we compare vanishing of first obstruction to efficient generation to vanishing of classical characteristic classes. We also study the secondary obstruction in the case that primary obstruction vanishes, and give more explicit conditions under which all rank $r$ projective modules on a smooth affine $k$-algebra of $\A^1$-cohomological dimension at most $d$ can be generated by $r+d-1$ or $r+d-2$ elements.

\subsection{Classical invariants}\label{subsec:classical-invariants}

 We now consider classical invariants that relate to obstruction theory for efficient generation of projective modules.
\begin{definition}\label{def:seg}Given a rank $j$ vector bundle $\xi$ on a smooth affine variety $X$ over a perfect field $k$, the {\em total Segre class} of $M$ is the inverse to the total Chern class in the Chow ring of $X$, and is written as $s(\xi)$.
The $i$-th Segre class $s_i(\xi)$ is the $i$-th graded piece of the total Segre class, which lies in $\CH^i(X)$, the Chow group of codimension $i$ cycles on $X$. \end{definition}
We follow \cite{AF5} for the next definition.
\begin{definition}\label{def:euler}The {\em Euler class} of a rank $j$ vector bundle $\xi$ on a smooth affine variety is the first obstruction to splitting a trivial bundle from $\xi$. 
The universal example is a class is $$\tilde{e}_r\in \CHW^j(\BGL_j,\det \gamma_j^{\vee}),$$
where $\gamma_j$ is the universal bundle on $\BGL_j$. Given a $\xi\: X \to \BGL_j$, \[e_r(\xi)=\xi^*(\tilde{e}_r) \in \CHW^j(X,\det \xi^{\vee}).\]
Given a smooth affine $k$-algebra $R$ and a projective module $M$ of rank $r$ over $R$ with associated vector bundle $\xi\: \spec R \to \BGL_r$, we define 
\[e_r(M)=e_r(\xi) \in \CHW^j(\spec R,\det \xi^{\vee} ).\] It is the first obstruction to splitting a copy of $R$ from $M$ as an $R$-module.
 \end{definition}
\begin{remark}\label{rmk:euler-lifts} 
The Euler class, as defined above, can be compared with numerous other constructions.  For oriented vector bundles on a smooth affine variety over a field $k$ having characteristic not equal to $2$, the Euler class as defined above, coincides with the Euler class in Chow-Witt theory, up to multiplication by a unit in the Grothendieck--Witt ring of $k$; this result is established in \cite[Theorem 1]{AF5}.  Additionally, there is a natural map $\CHW^j(\BGL_j, \operatorname{det} \gamma_j^\vee) \to \CH^j(\BGL_j)$ under which the Euler class maps to the usual top Chern class \cite[Proposition 5.8]{AF5}.
\end{remark}

\subsection{Efficient generation of line bundles and powers of the first Chern class}\label{subsec:generation-line-bundles}

Let $k$ be an algebraically closed field. In \cite[Corollary 3.16]{Mu}, Murthy shows that a rank $1$ module $L$ over a smooth affine $k$-algebra $R$ of dimension $d$ can be generated by $d$ elements if and only if the $d$-th power of the first Chern class of $L$ vanishes. In general, some hypothesis on $k$ will be necessary, but $k$ need not be algebraically closed and also need not satisfy other technical conditions listed in \cite[Theorem 1.8]{Mu}. 
\begin{theorem}\label{thm:line-bundle-generation} Let $X=\spec R$ be a smooth affine variety of dimension at most $d\geq 2$ over a field of $2$-cohomological dimension at most $1$. Let $L$ be a rank $1$ projective module over $R$. Then $L$ can be generated by $d$ elements if and only if the $d$-th power of the first Chern class of $L$ is zero in $\CH^d(X).$
\end{theorem}
\begin{proof} First consider $k$ a general field. In the case $r=1$, \Cref{prop:fib_stiefel_pf} gives a fiber sequence

\[\A^{d}\setminus 0 \to \mathbb P^{d-1}_k \xrightarrow{f}  \BGL_1.\]
Note that $\CH^*(\BGL_1) \cong \mathbb Z[H]$ where $H$ is in degree $1$, while 
$\CH^*(\mathbb P^{d-1}_k)\cong \mathbb Z[H]/H^d$.
 The map $f$ induces the quotient map. 
Thus, if $L\: X\to \BGL_1$ lifts to $\mathbb P^{d-1}_k$, $c_1(L)^d=0$. 

On the other hand, the first potentially nontrivial Moore--Postnikov invariant for the morphsim $\mathbb P^{d-1} \to \BGL_1$ takes the form $m\: \BGL_1 \to K^{\mathbb G_m}(\KMW_d,d)$. Let $\xi\: X \to \BGL_1$ classify $L$. Note that, since $X$ has $\A^1$-cohomological dimension at most $d$, $X$ lifts to $\mathbb P^{d-1}$ if and only if $\xi^*(m)=0$. Let $\gamma_1$ classify the universal bundle on $\GL_1$. 

We obtain a commutative diagram
\[ \begin{tikzcd}[row sep=1.5em, column sep =1.5em] \CHW^d(\BGL_1, \gamma_1 ) \ar[r]\ar[d]& \CHW^d(X,L) \ar[d]\\
\CH^d(\BGL_1) \ar[r]& \CH^d(X),
\end{tikzcd}
\]
where both horizontal maps are induced by $\xi$ and both vertical maps are induced by the natural morphism $\KMW_d \to \KM_d$. Under the hypothesis that the $2$-cohomological dimension of $k$ is at most $1$, the right-hand vertical map is an isomorphism \cite[Proposition 5.2]{AF2}. Therefore the obstruction to lifting to $\xi$ to $\mathbb P^{d-1}$ factors through the image of the Moore--Postnikov invariant $m$ in $\CH^d(\BGL_1).$ Note that the image of $m$ in $\CH^d(\BGL_1)\cong \mathbb Z\{ H^d\}$ is nonzero, and is therefore a multiple of the $d$-th power of $H$. Therefore the only obstruction to lifting $\xi$ to $\mathbb P^{d-1}$ is a nonzero multiple of $c_1(L)^d$.
\end{proof}

\begin{remark} The forward implication of \Cref{thm:line-bundle-generation} also follows from \cite[Lemma 5.4]{FFR25}, which in fact applies more generally to $X$ a smooth $k$-variety and $k$ an arbitrary field. \end{remark} 

\subsection{The first nontrivial obstruction for rank at least $2$ and conditions for $r+d-1$ generation}\label{subsec:first-obstr-classical}
Our goal is to study rank $r$ projective modules on smooth affine varieties of $\A^1$-cohomological dimension at most a given $d$, but we begin with a more general result.

\begin{proposition}\label{thm:first-obstr} Let $X=\spec R$ be a smooth affine variety over a perfect field. Let $M$ be a projective $R$-module of rank $r \geq 2$. Suppose that $M$ can be generated by $n+1$ elements where $n\geq r+2$. Let $Q$ be a rank $n-r+1$ module such that $M \oplus Q \cong R^{n+1}$. 
\begin{itemize}
\item  If $n-r$ is odd, the $o_{n-r,n,r}(M)$ vanishes if and only if $s_{n-r+1}(M)$ vanishes, where the latter denotes the top Segre class of $M$ as in \Cref{def:seg}.
\item If $n-r$ is even,  $o_{n-r,n,r}(M)$ vanishes if and only if $e_{n-r+1}(Q)$ vanishes, where the latter denotes the Euler class of $Q$ as in \Cref{def:euler}. In particular, vanishing of $e_{n-r+1}(Q)$ is independent of the choice of $Q$.
\end{itemize}
\end{proposition} 
\begin{proof} Consider \Cref{diag:big_diagram}:
\begin{equation}\label[diagram]{diag:big_diagram3}
\begin{tikzcd}[row sep=1.5em, column sep =1.5em]
& & X \ar[d]\ar[dl,dashed]\ar[dd, bend left=50, "h"]\\
\ST{r}{n}\ar[r]& \Gr{r}{n}\ar[r] & \BGL_r \\
(\A^{n-r+1}\setminus 0)\times \Omega (\A^{n+1}\setminus 0) \ar[r]\ar[u,"f"]\ar[d,"\operatorname{pr}_1"] & \Gr{r}{n} \ar[r]\ar[u]\ar[d]& \Gr{r}{n+1}\ar[u] \ar[d]\\
\A^{n-r+1}\setminus 0 \ar[r]  \ar[r] & \BGL_{n-r} \ar[r] &\BGL_{n-r+1},
\end{tikzcd}
\end{equation}
where $\operatorname{pr}_1$ is projection onto the first factor and $h$ represents the surjection $M \oplus Q \to M$ from a free module of rank $n+1$ onto $M$. Our goal is to understand the first potentially nontrivial obstruction to the existence of a dashed arrow, which is the obstruction to lifting to the $(n-r)$-th stage of the Moore--Postnikov tower.

We now relate the first obstructions in the Moore--Postnikov towers for the morphisms \[\Gr{r}{n} \to \BGL_r, \,\,\, \Gr{r}{n} \to \Gr{r}{n+1}, \text{  and  }\BGL_{n-r} \to \BGL_{n-r+1}.\] First, note that $\operatorname{pr}_1$ induces an isomorphism on $\pia_{i}$ for $i<n-1$ and therefore for $i<n-r+1$ since $r\geq 2$. Given $h\: X \to \Gr{r}{n+1}$, \Cref{diag:big_diagram3} shows that the first potentially nontrivial obstruction in the lifting problem
\begin{equation}\label{lifting-1}
\begin{tikzcd}[row sep=1.5em, column sep =1.5em]
 & X \ar[d,"h"]\ar[dl,dashed]\\
\Gr{r}{n} \ar{r} &\Gr{r}{n+1}
\end{tikzcd}
\end{equation}
is identified with that for 
\begin{equation}\label{lifting-2}
\begin{tikzcd}[row sep=1.5em, column sep =1.5em]
& X \ar[d,"N"]\ar[dl,dashed]\\
\BGL_{n-r}\ar{r} &\BGL_{n-r+1}.
\end{tikzcd}
\end{equation}
By definition, the latter obstruction is precisely $e_{n-r+1}(Q)$. 

We now relate $f$ from \Cref{diag:big_diagram3} to maps we have already understood. Note that $f \cong a_{r}^{n-r} \circ \operatorname{pr}_{1}$, with notation from \Cref{diag:abnotation}.
So, by \Cref{cor:general-first-sheaf-compuation}, $f$ induces an isomorphism on $\pia_{n-r}$ for $n-r$ even and is the quotient of $\KMW_{n-r+1}$ by $\eta$ when $n-r$ is odd. 
Thus, for any parity of $n-r$, we find that $o_{n-r,n,r}(M)$ vanishes if $ e_{n-r+1}(Q)=0$.  If $n-r$ is even, the converse is also true. In the case that $n-r$ is odd, $o_{n-r,n,r}(M)$ is a unit multiple of $e_{n-r+1}(Q)$ modulo $\eta$, which is a unit multiple of the top Chern class of $Q$. Therefore $o_{n-r,n,r}(M)=0$ if and only if the top Segre class of $M$ vanishes.
\end{proof} 
\begin{remark}\label{rmk:rr1-discussion} In the statement of \Cref{thm:first-obstr} with $n=r+1$, the Moore--Postnikov framework outlined in  \cite[Section 6.1]{AF-splitting} does not apply as stated to the study the problem of lifting along $\Gr{r}{r+1} \to \BGL_r$, since the fiber $\ST{r}{r+1}$ is not simply connected. However, we can import the results of \cite[Section 4]{Robinson72} to the motivic setting. The action of $\pia_1\BGL_r\cong \mathbb G_m$ on $\pia_1\ST{r}{r+1} \cong \mathbb G_m$ arising from the fibration $\ST{r}{r+1} \to \Gr{r}{r+1} \to \BGL_r$ is trivial, the first obstruction group is $\CH^2(X)$, and the first obstruction can again be identified with the Segre class. Given this, the remaining results in this section (\Cref{thm:different-murthy,cor:11,cor:murthy-segre-suff-quad-closed}) apply when $d=2$, and those in the next section (\Cref{thm-rd2,thm-rd2-2,thm:rd2-alg-closed-char-neq-2}) apply when $d=3$. These modifications also prove the homotopy Forster--Swan theorem (\Cref{htpy:Fo-Sw}) when $d=1$.
\end{remark}

If we impose dimension hypotheses on $X$ from \Cref{thm:first-obstr}, we obtain a version of Murthy's theorem \cite{Mu} over a not necessarily algebraically closed field.
\begin{theorem}\label{thm:different-murthy}Let $X=\spec R$ be a smooth affine variety of $\A^1$-cohomological dimension at most $d \geq 3$ over a perfect field $k$.  Let $M$ be a rank $r$ projective module over $R$, where $r\geq 2$. 
\begin{itemize}
\item If $d$ is even, $M$ can be generated by $d+r-1$ generators if and only if the top Segre class of $M$ vanishes.
\item If $d$ is odd and $M \oplus Q \simeq R^{d+r}$ for $Q$ a projective module of rank $d$, then $M$ can be generated by $d+r-1$ generators if and only if the Euler class of $Q$ is zero in $\CHW^d(X,\det M).$ Such an $Q$ always exists, and this condition is independent of the choice of $Q$.
\end{itemize} 
\end{theorem}

\begin{proof}
We take $n=r+d-1$ in \Cref{thm:first-obstr}. By dimensional considerations, the first obstruction is the only obstruction to the lifting problem in question.
\end{proof}
This immediately implies that, under cohomological vanishing assumptions, all projective modules of rank $r$ on certain smooth affine $d$-folds can be efficiently generated.

\begin{corollary}\label{cor:11}Let $X=\spec R$ be a smooth affine variety of $\A^1$-cohomological dimension at most $d\geq 3$ over a perfect field and let $r\geq 2$. \begin{itemize}
\item If $d$ is even and $\CH^d(X)=0$, then every rank $r$ projective module over $R$ can be generated by $r+d-1$ elements.
\item If $d$ is odd and $\CHW^d(X,\mathcal L)=0$ for any line bundle $\mathcal L$ on $X$, then every rank $r$ projective module over $R$ can be generated by $r+d-1$ elements.
\end{itemize}
\end{corollary}
With hypotheses on the base field, we obtain a stronger result.
\begin{corollary}\label{cor:murthy-segre-suff-quad-closed} Let $X=\spec R $ be a smooth affine variety of dimension at most $d\geq 3$ over a perfect field $k$. Suppose also that $k$ has $2$-cohomological dimension at most $1$. A projective module $M$ of rank $r \geq 2$ can be generated by $r+d-1$ elements as an $R$-module if and only if the Segre class $s_d(M)$ vanishes.
\end{corollary}
\begin{proof} By \Cref{thm:different-murthy}, it suffices to consider the case $d$ odd. If $k$ has $2$-cohomological dimension at most $1$, the natural map $\KMW_{d} \to \KM_{d}$ induces, for any line bundle $\mathcal L$ on $X$, an isomorphism
$\CHW^d(X,\mathcal L) \cong \CH^d(X)$ \cite[Proposition 5.2]{AF2}. Under this identification, the Euler class of $Q$ is a unit multiple of the top Chern class. In particular, the Euler class of a complementary bundle $Q$ is a unit multiple of the Segre class of $M$.
\end{proof}
\begin{remark} If $k$ is algebraically closed and $R$ is a $k$-algebra of dimension $d$, then the previous corollary can be deduced from celebrated work of Murthy \cite[Corollary 3.15]{Mu}.  On the other hand, note, for example, that \Cref{cor:murthy-segre-suff-quad-closed} applies when $k$ is taken to be a finite field, since such fields have $2$-cohomological dimension equal to $1$.
\end{remark}

In combination with results from \cite{AFL25}, we also deduce efficient generation results for real algebraic vector bundles. 

\begin{theorem}
Let $X = \spec R$ be a smooth affine algebra of dimension $d\geq 3$ over $\mathbb{R}$. Let $M$ be a rank $r\geq 2$ projective $R$-module over $R$. Then M can be generated by $d+r-1$ elements if and only if the top Segre class of M vanishes.
 \end{theorem}
\begin{proof}The statement for $n$ even follows from \Cref{thm:different-murthy}. For $n$ odd, we appeal to \cite[2.1.1 and 2.1.2(3)]{AFL25}. 
\end{proof}

\subsection{Further reduction of the number of generators and secondary obstructions}\label{further-reduction}
Assuming that a module can be generated by $r+d-1$ elements, we study conditions for $(r+d-2)$-generation.
\begin{proposition}\label{thm-rd2}Let $X=\spec R$ be a smooth affine variety of $\A^1$-cohomological dimension at most  $d\geq 4$ over a perfect field. Suppose that $M$ is a rank $r\geq 2$ projective module over $R$ generated by $r+d-1$ elements and that $H^d(X,\pia_{d-1}\ST{r}{r+d-2}(\mathcal L))=0$ for any line bundle $\mathcal L$ on $X$. Then:
\begin{itemize}
\item If $d$ is odd, $M$ can be generated by $d+r-2$ elements if and only if the Segre class $s_{d-1}(M)$ vanishes.
\item Let $Q$ be a projective module of rank $d-1$ such that $M\oplus Q$ is free.  If $d$ is even, then $M$ can be generated by $d+r-2$ elements if and only if $e_{d-1}(Q)=0$. This condition is independent of the choice of $Q$.
\end{itemize} 
\end{proposition}
\begin{proof} By \Cref{thm:first-obstr} with $n=r+d-2$ and $j=r$, the first obstruction to reducing the number of generators is 
$s_{d-1}(M)$ for $d$ odd and $e_{d-1}(Q)$ for $d$ even. 
The second obstruction is valued in $H^d(X,\pia_{d-1}\ST{r}{r+d-2}(\mathcal L))=0$, where $\mathcal L$ the determinant of $M$. 
\end{proof}
The above proposition simplifies under the additional hypothesis that the base field is quadratically closed, following the same proof as 
\Cref{cor:murthy-segre-suff-quad-closed}.
\begin{corollary}\label{thm-rd2-2}Let $X=\spec R$ be a smooth affine variety of $\A^1$-cohomological dimension at most 
$d \geq 4$ over a perfect quadratically closed field. Suppose that $M$ is a rank $r\geq 2$ projective module generated by $r+d-1$ elements over $R$. 
Suppose also that $H^d(X,\pia_{d-1}\ST{r}{r+d-2}(\mathcal L))=0$ for any line bundle $\mathcal L$ on $X$. 
Then $M$ can be generated by $d+r-2$ elements if and only if the Segre class $s_{d-1}(M)$ vanishes.
\end{corollary}
In the case of an algebraically closed base field, we prove a sharper result.
\begin{theorem}\label{thm:rd2-alg-closed-char-neq-2} Let $X = \spec R$ be a smooth affine variety of dimension at most $d\geq 4$ over an algebraically closed field $k$. Suppose also that $H^d(X,\pia_{d-1}(\A^{d-1}\setminus 0))=0$. Let $M$ be a projective module of rank $r\geq 2$. Then $M$ can be generated by $r+d-2$ elements if and only if $s_{d}(M)=0$ and $s_{d-1}(M)=0$.
\end{theorem}
\begin{proof}
 Appealing to \Cref{thm:first-obstr} and \cite[Corollary 5.3]{AF2}, the first possibly nontrivial obstruction to lifting 
\begin{equation}\label[diagram]{diag:basic-lift-a}\begin{tikzcd}[row sep=1.5em, column sep =1.5em]
& X \ar[d]\ar[dl,dashed]\\
\Gr{r}{r+d-2} \ar[r]& \BGL_r \end{tikzcd}\end{equation}
is precisely $s_{d-1}(M)$. Consider the second possibly nontrivial obstruction. 
We have the following diagram of fiber sequences:
\begin{equation}\label[diagram]{diag:second-osbtr}\begin{tikzcd} [row sep=1.5em, column sep =1.5em]
\ST{r}{r+d-2} \ar[r,"g"] \ar[d]& \ST{r}{r+d-1}\ar[d]\\ 
\Gr{r}{r+d-2} \ar[r] \ar[d,"A"]& \Gr{r}{r+d-1}\ar[d,"B"] \\
 \BGL_r \ar[r] & \BGL_r.
\end{tikzcd}
\end{equation}
We are interested in lifting along the map $A$; we compare to lifting along $B$. Given a lift to the 
$(d-1)$-st stage of the Moore--Postnikov tower for the morphism $A$, the obstruction to lifting to the $d$-th stage in the Moore--Postnikov tower for 
$A$ maps to the obstruction to lifting to the 
$d$-th stage of the Moore--Postnikov tower for $B$. This map of obstructions is induced by the map
\[g_*:=H^d(X,\pia_{d-1}g): H^d(X,\pia_{d-1}{\ST{r}{r+d-2}}(\mathcal L)) \to H^d(X,\pia_{d-1}{\ST{r}{r+d-1}}(\mathcal L))\] on cohomology, where $\mathcal L=\det M$. We use analogous notation for other induced maps on cohomology.
We will show that $g_*$ is injective. Given this, the remaining obstruction to lifting along $A$ is precisely the first and only nontrivial obstruction to lifting along $B$, which by \Cref{cor:murthy-segre-suff-quad-closed} is $s_{d}(M)$.

Note that the morphism $g$ can be factored as follows:
\begin{equation}\label[diagram]{diag:second-osbt2}\begin{tikzcd} 
 \ST{r}{r+d-2} \ar[r,"b_{d-2}^r"]\ar[dr,"g"]& \ST{r-1}{r+d-2}\ar[d]\\
 & \ST{r}{r+d-1},
\end{tikzcd}
\end{equation}
where $b_{d-2}^r$ is as in \Cref{diag:abnotation} with fiber $\A^{d-1}\setminus 0$. 
We claim that the vertical morphism in \Cref{diag:second-osbt2} induces an isomorphism after applying 
$H^d(X,\pia_{d-1}(-)(\mathcal L)).$ For 
$r \geq 3$, this is by \Cref{cor:stable-range}. For 
$r=2$, the claim follows from \Cref{pirr2} and \cite[Proposition 5.2]{AF2}. 

Now, consider the following commutative diagram, derived from \Cref{diag:second-osbt2}:

\begin{equation}\label[diagram]{diag:injectivity}
\begin{tikzcd} [row sep=1.5em, column sep =1.5em]
H^d(X,\pia_{d-1}(\A^{d-1}\setminus 0)(\mathcal L)) \ar[r]& H^d(X,\pia_{d-1}\ST{r}{r+d-2}(\mathcal L))\ar[d,"{(b_{d-2}^r)_*}"]\ar[dd,"g_*", bend left=90] \\ 
 & H^d(X,\pia_{d-1}\ST{r-1}{r+d-2}(\mathcal L))\ar[d,"\simeq"]\\
 & H^d(X,\pia_{d-1}\ST{r}{r+d-1}(\mathcal L)).
\end{tikzcd}
\end{equation}
We claim the image of the horizontal arrow in \Cref{diag:injectivity} contains the kernel of $g_*$. Given this, if 
$H^d(X,\pia_{d-1}(\A^{d-1}\setminus 0))=0$, then by \cite[Lemma 2.2.3]{fasel2021suslin} the map $g_*$ is injective.
To prove the claim, consider \Cref{prop:necessary-computation} with $j=d-2$, $n=r+d-2$. If 
$d$ is even, we see that the kernel of the map on $d$-th cohomology induced by 
$\pia_{d-1}b_{d-2}^r$ is $H^d(X,\operatorname{Im}(\pia_{d-1}a_{d-2}^r)(\mathcal L))$ which is a quotient of 
$H^d(X,\pia_{d-1}(\A^{d-1}\setminus 0)(\mathcal L)).$
If $d$ is odd,  consider again \Cref{prop:necessary-computation}. Consider the sequence
\begin{equation}\label[diagram]{exact-stuff11}\begin{tikzcd}[row sep=1.5em, column sep =1.5em] 
& & H^{d-1}(X,\mathcal I^{d}(\mathcal L))\ar[d]\\
H^d(X,\operatorname{Im}(\pia_{d-1}a_{d-2}^r)(\mathcal L)) \ar[r]
& H^d(X,\pia_{d-1}\ST{r}{r+d-2}(\mathcal L)) \ar[r] \ar[dr,"(b_{d-2}^r)_*\,\,\,\,\,\,\," below]
& H^d(X,2\KM_{d}(\mathcal L)) \ar[d]\\
& & H^d(X, \KMW_d(\mathcal L))\ar[d]\\
& & H^d(X,\mathcal I^{d}(\mathcal L))\end{tikzcd}
\end{equation}
where $\mathcal I^{d}$ is the $d$-th power of the fundamental ideal in the Witt ring. Both the horizontal row and vertical column in \Cref{exact-stuff11} are exact. 
Since $k$ is algebraically closed and $d\geq 4$, \cite[Proposition 5.2]{AF2} implies that the middle vertical morphism is an isomorphism.
\end{proof}
\begin{remark} \Cref{thm:rd2-alg-closed-char-neq-2} gives another proof of \Cref{char-zero-segre-vanish-rd2}: by the proof of \cite[Theorem 7.1.1]{ABH}, if $k$ is algebraically closed of characteristic zero then $H^d(X,\pi_{d-1}(\A^{d-1}\setminus 0))=0$. 
\end{remark} 

\section{A Forster--Swan theorem for symplectic modules}\label{sec:sym-forster-swan}
Let $X=\spec R$ be a smooth affine variety over a perfect field $k$ and let 
$n \geq r+1$. Given an $R$ module $M$ of rank $2r$ equipped with a nondegenerate symplectic form 
$\omega\: M \otimes M \to R$, one might seek an efficient generation result in the symplectic setting. 
The symplectic analogues of free modules are hyperbolic modules, i.e., direct sums of copies of the rank 
$2$ symplectic module $H$ given by $R^2$ with the form associated to the matrix 
\[\begin{pmatrix} 0 & 1 \\ -1 & 0 \end{pmatrix}.\]  
As such, we might ask for the minimal $k$ such that $M$ is a direct summand of $H^{\oplus k}$. 
Let $\xi\: X \rightarrow \BSp_{2r}$ represent a symplectic $R$-module of rank $2r$. The universal rank $2r$ symplectic bundle on $\Grsp{2r}{2n}$ is represented by a morphism
\begin{equation}\label[diagram]{symp-fib-seq1}\Grsp{2r}{2n} \rightarrow \BSp_{2r},\end{equation} and a lift of $\xi$ to a map $\tilde{\xi}\: X \to \Grsp{2r}{2n}$ 
corresponds to a presentation of the symplectic module $\xi$ as a direct summand of a rank 
$2n$ hyperbolic symplectic $R$-module.  To see when such a presentation exists, we consider the Moore--Postnikov factorization of \Cref{symp-fib-seq1}. This yields a symplectic Forster--Swan theorem:

\begin{theorem}\label{thm:symplectic-F-S}
Let $X = \spec R$ be a smooth affine variety of $\A^1$-cohomological dimension at most 
$d \geq 2$ over a perfect field $k$. Let $M$ be a symplectic module of rank $2r$. \begin{itemize}
\item If $d$ is even, then $M$ is a direct summand of a hyperbolic symplectic $R$-module of rank $2r+d$.
\item If $d$ is odd, then $M$ is a direct summand of a hyperbolic symplectic $R$-module of rank $2r+d-1.$
\end{itemize}
\end{theorem}
\begin{proof} 
For $d$ even, set $2n=2r+d$.  For $d$ odd, set $2n=2r+d-1.$ 
Given a lift of $M\: X \to \BSp_{2r}$ to the $(i-1)$-st stage of the Moore-Postnikov factorization for \Cref{symp-fib-seq1}, the obstruction to lifting to the 
$i$-th stage is an element in $ H^{i+1} (X, \pia_{i} \STsp{2r}{2n})$ (cf. \Cref{prop:symplectic_fiber_seqs}) . If $i \geq d$, then this obstruction to lifting vanishes automatically. 
For $i< d$, $\pia_{i} \STsp{2r}{2n}=0$ by \Cref{cor:connectivity-sp}.
\end{proof}
As in the case of finitely generated projective modules, it is natural to ask under which circumstances the estimate from 
\Cref{thm:symplectic-F-S} above can be improved. The following elementary fact will be useful:

\begin{lemma}\label{lem:elementary} Let $(N,\omega)$ be a symplectic module of rank $2r$ over a commutative ring $R$. Suppose that 
$Q$ splits a rank $1$ summand as an $R$-module. Then $Q$ splits a hyperbolic module as a symplectic $R$-module.
\end{lemma}
\begin{proof} Let $Q^{\vee}$ denote the $R$-linear dual of $Q$ and let
\[\omega^{\#}\: Q \to Q^{\vee}\] denote the isomorphism of $Q$ with $Q^{\vee}$ associated to 
$\omega$.  The hypothesis that $Q$ splits off a trivial module is equivalent to the existence of a surjection of $R$-modules $\phi\: Q \to R.$
Let $\alpha \in Q$ be a preimage of $1\in R$. Let 
$\beta \in Q$ be the preimage of 
$\phi \in Q^{\vee}$ under $\omega^{\#}$. Since $\omega(\alpha,\beta)=1$, 
we find that the submodule of $Q$ generated by $\alpha$ and $\beta$ is hyperbolic.
The $\omega$-completement of the submodule generated by $\alpha$ and $\beta$ in $Q$ is a symplectic submodule of 
$Q$ that is of rank $d-2$ as an $R$-module. Call this symplectic module $Q'$. We find that $Q\simeq Q'\oplus H$, 
where $H$ denotes the rank $2$ hyperbolic symplectic module.
\end{proof} 
\begin{theorem}\label{thm:first-obstruction-even} Let $X = \spec R$ be a smooth affine variety of even $\A^1$-cohomological dimension at most 
$d \geq 2$ over a perfect field $k$. Let $M$ be a symplectic module of rank $2r$. Let $Q$ be any module such that $M\oplus Q$ is hyperbolic of rank $2r+d$. 
Then $M$ is a summand of a hyperbolic module of rank $2r+d-2$ if and only if $e_d(Q)$ is zero in $\CHW^d(X)$. 
This condition is independent of the choice of $Q$.
\end{theorem}
\begin{proof}
The assumption $e_d(Q)=0$ implies that $Q$ splits off a trivial rank $1$ summand, so 
\Cref{lem:elementary} applies. We have that $M \oplus Q' \oplus  H \simeq H^{2r+d}$ for $Q'$ symplectic of rank $d-2$. 
By considering Moore--Postnikov obstruction theory for the map $\BSp_{2r+d-2} \to \BSp_{2r+d}$ and the fiber sequence
\[\A^{2r+d}\setminus 0 \to \BSp_{2r+d-2} \to \BSp_{2r+d}\] and using that $R$ has dimension $d$ over $k$, we deduce that 
$M \oplus Q' \simeq H^{2r+d-2}.$ 
\end{proof}
We immediately obtain a few consequences:
\begin{corollary}
Let $X = \spec R$ be a smooth affine variety of even $\A^1$-cohomological dimension at most $d \geq 2$ over a perfect field $k$. If $\CHW^{d}(X) = 0$, then any symplectic $R$-modules of rank $2r$ is a direct summand of a hyperbolic symplectic $R$-module of rank $2r+d-2$.
\end{corollary}

\begin{corollary}\label{d-segre-obstruction-symplectic} Let $k$ be a perfect field of $2$-cohomological dimension at most $1$. 
Let $X=\spec R$ be a smooth affine variety of dimension at most $d$ over $k$, and let $M$ be a symplectic module of rank $2r$ over $R$. 
If $d$ is even, then $M$ is a summand of a symplectic module of rank $2r+d-2$ if and only if the Segre class $s_{d}(M)$ vanishes.
\end{corollary}
\begin{proof}Consider the set-up as in \Cref{thm:first-obstruction-even}. Note that, by \Cref{rmk:euler-lifts} and \cite[Proposition 5.2]{AF2}, 
$s_d(M)=0$ if and only if $e_d(Q)=0$.  \end{proof}

We next consider symplectic modules over odd-dimensional varieties.
\begin{theorem}\label{thm:symplectic-odd-further-reduction}
Let $X = \spec R$ be a smooth affine variety of odd $\A^1$-cohomological dimension at most $d\geq 3$ over a perfect field $k$. 
Let $M$ be a rank $2r\geq 4$ symplectic module and let $Q$ be a symplectic module of rank $d-1$ such that $Q\oplus M$ is hyperbolic. 
If $H^{d}(X, \pia_{d-1}(\mathbb{A}^{d-1}\setminus 0)) = 0$ and $e_{d-1}(Q)=0$ in $\CHW^{d-1}(X)$, then $M$ can be generated by $2r+d-3$ elements.
\end{theorem}
\begin{proof}
The first obstruction to splitting a copy of $R$ from $Q$ is the Euler class $e_{d-1}(Q)$. 
The second obstruction lies in $H^d(X,\pia_{d-1}(\A^{d-1}\setminus 0))=0$. By \Cref{lem:elementary}, we conclude that 
$M\oplus Q'\oplus H $ is hyperbolic of rank $2r+d-1$. Again, by considering Moore--Postnikov obstruction theory for the fiber sequence
\[ \A^{2r+d-1}\setminus 0 \to \BSp_{2r+d-3} \to \BSp_{2r+d-1},\]
we find that $M \oplus Q'$ is hyperbolic of rank $2r+d-3$.
\end{proof}
\begin{remark} Without the hypothesis that $H^{d}(X,\pia_{d}(\A^{d}\setminus 0))=0$, the proof of \Cref{thm:symplectic-odd-further-reduction} 
shows that $e_{d-1}(Q)$ is the first obstruction to $M$ being a summand of a rank $2r+d-3$ hyperbolic module. \end{remark}
Under additional hypotheses on the field, the previous result simplifies.
\begin{corollary}\label{cor:symplectic-even-further-reduction}
Let $X = \spec R$ be a smooth affine variety of odd dimension at most $d \geq 3$ over a perfect field $k$. 
Let $M$ be a rank $2r\geq 4$ symplectic $R$-module. \begin{itemize} 
\item If $k$ is quadratically closed and $H^d(X,\pia_{d-1}(\A^{d-1}\setminus 0))=0$, 
$M$ is a direct summand of a hyperbolic module of rank $2r+d-3$ elements if and only if $s_{d-1}(M)=0$.

\item If $k$ is an algebraically closed field of characteristic zero, $M$ is a direct summand of a hyperbolic module of rank $2r+d-3$ elements if and only if $s_{d-1}(M)=0$.
\end{itemize}
\end{corollary}
\begin{proof}
 By \cite[Corollary 5.3]{AF2}, $\CHW^{d-1}(X) \cong \CH^{d-1}(X)$ for $k$ quadratically closed so the result follows. 
If $k$ is additionally algebraically closed of characteristic zero,  the proof of \cite[Theorem 7.1.1]{ABH} shows that $H^{d}(X,\pia_{d-1}(A^{d-1}\setminus 0)) = 0$.
\end{proof}
\bibliographystyle{abbrv}
\bibliography{arxiv2}

\begin{thebibliography}{10}

\bibitem{ABHWhitehead}
A.~Asok, T.~Bachmann, and M.~Hopkins.
\newblock On the {W}hitehead theorem for nilpotent motivic spaces.
\newblock {\em arXiv preprint}, abs/2210.05933, 2022.
\newblock Version posted October 12, 2022.

\bibitem{ABH}
A.~Asok, T.~Bachmann, and M.~J. Hopkins.
\newblock On {$\mathbb{P}^1$}-stabilization in unstable motivic homotopy
  theory.
\newblock {\em arXiv preprint}, abs/2306.04631, 2023.
\newblock Version posted June 12, 2024.

\bibitem{AF14}
A.~Asok and J.~Fasel.
\newblock Algebraic vector bundles on spheres.
\newblock {\em Journal of Topology}, pages 894 -- 926, 2014.

\bibitem{AF2}
A.~Asok and J.~Fasel.
\newblock A cohomological classification of vector bundles on smooth affine
  threefolds.
\newblock {\em Duke Math. Journal}, 163(14), 2014.

\bibitem{AF-splitting}
A.~Asok and J.~Fasel.
\newblock {S}plitting vector bundles outside the stable range and
  {$\mathbb{A}^1$}-homotopy sheaves of punctured affine spaces.
\newblock {\em J. Amer. Math. Soc.}, 28(4):1031--1062, 2015.

\bibitem{AF5}
A.~Asok and J.~Fasel.
\newblock Comparing {E}uler classes.
\newblock {\em Quarterly. J. Math.}, 67:603--635, 2016.

\bibitem{AF4}
A.~Asok and J.~Fasel.
\newblock Algebraic vs. topological vector bundles on spheres.
\newblock {\em J. Ramanujan Math. Soc.}, 32(3):201--216, 2017.

\bibitem{AFL25}
A.~Asok, J.~Fasel, and S.~Lerbet.
\newblock Splitting vector bundles over real algebraic varieties, 2025.

\bibitem{AHW1}
A.~Asok, M.~Hoyois, and M.~Wendt.
\newblock {A}ffine representability results in {$\mathbb{A}^1$}-homotopy theory
  {I}: vector bundles.
\newblock {\em Duke Math. J.}, 166(10):1923--1953, 2017.

\bibitem{AHW2}
A.~Asok, M.~Hoyois, and M.~Wendt.
\newblock Affine representability results in {$\mathbb A^1$}-homotopy theory,
  {I}{I}: {P}rincipal bundles and homogeneous spaces.
\newblock {\em Geometry \& Topology}, 22(2):1181--1225, 2018.

\bibitem{stronglystrictly}
T.~Bachmann.
\newblock {S}trongly {$A^1$}-invariant sheaves (after {F}. {M}orel).
\newblock {\em arXiv preprint}, abs/2406.11526, 2024.
\newblock Version posted June 17, 2024.

\bibitem{add5}
S.~Banerjee, C.~Bhaumik, and H.~P. Sarwar.
\newblock Efficient generation, unimodular element in a geometric subring of a
  polynomial ring, 2023.

\bibitem{add3}
M.~K. Das.
\newblock The euler class group of a polynomial algebra.
\newblock {\em Journal of Algebra}, 264(2):582--612, 2003.

\bibitem{add4}
M.~K. Das and M.~A. Zinna.
\newblock Efficient generation of ideals in overrings of polynomial rings.
\newblock {\em Journal of Pure and Applied Algebra}, 219(9):4016--4034, 2015.

\bibitem{eisenbudevans73}
D.~Eisenbud and E.~G. Evans~Jr.
\newblock Generating modules efficiently: theorems from algebraic {$K$}-theory.
\newblock {\em Journal of Algebra}, 27(2):278--305, 1973.

\bibitem{fasel2021suslin}
J.~Fasel.
\newblock Suslin's cancellation conjecture in the smooth case.
\newblock {\em arXiv preprint}, abs/2111.13088, 2021.
\newblock Version posted December 10, 2024.

\bibitem{FFR25}
U.~A. First, M.~Florence, and Z.~Rosengarten.
\newblock Algebraic groups with torsors that are versal for all affine
  varieties, 2025.

\bibitem{Fo}
O.~Forster.
\newblock {\"U}ber die {A}nzahl der {E}rzeugenden eines ideals in einem
  noetherschen {R}ing.
\newblock {\em Math. Z.}, 84:80--87, 1964.

\bibitem{add7}
R.~Jebasingh and M.~A. Zinna.
\newblock Efficient generation of ideals and an analogue of a result of mohan
  kumar.
\newblock {\em Journal of Algebra}, 647:567--583, 2024.

\bibitem{add6}
M.~K. Keshari and M.~A. Zinna.
\newblock Efficient generation of ideals in a discrete hodge algebra.
\newblock {\em Journal of Pure and Applied Algebra}, 221(4):960--970, 2017.

\bibitem{add2}
N.~M. Kumar.
\newblock On two conjectures about polynomial rings.
\newblock {\em Inventiones mathematicae}, 46(3):225--236, 1978.

\bibitem{Morel}
F.~Morel.
\newblock {\em $\mathbb A^1$-{A}lgebraic {T}opology over a {F}ield}, volume
  2052 of {\em Lecture Notes in Mathematics}.
\newblock Springer, New York, 2012.

\bibitem{MV99}
F.~Morel and V.~Voevodsky.
\newblock {$\mathbb A^{1}$}-homotopy theory of schemes.
\newblock {\em Publications Math{\'e}matiques de l'IH{\'E}S}, 90:45--143, 1999.

\bibitem{Mu}
M.~P. Murthy.
\newblock Zero cycles and projective modules.
\newblock {\em Ann. Math.}, 140(2):405--434, 1994.

\bibitem{Robinson72}
C.~A. Robinson.
\newblock {M}oore--{P}ostnikov systems for non-simple fibrations.
\newblock {\em Illinois Journal of Mathematics}, 16(2):234--242, 1972.

\bibitem{add1}
A.~Sathaye.
\newblock On the {F}orster--{E}isenbud--{E}vans conjectures.
\newblock {\em Inventiones mathematicae}, 46(3):211--224, 1978.

\bibitem{suslin77a}
A.~Suslin.
\newblock A cancellation theorem for projective modules over algebras.
\newblock {\em Doklady Akademii nauk SSSR}, 236:808--811, 1977.

\bibitem{SwanProjModVB}
R.~Swan.
\newblock Vector bundles and projective modules.
\newblock {\em Transactions of the American Mathematical Society},
  105(2):264--277, 1962.

\bibitem{Sw}
R.~G. Swan.
\newblock The number of generators of a module.
\newblock {\em Mathematische Zeitschrift}, 102:318--322, 1967.

\end{thebibliography}
\end{document}